\documentclass[11pt,reqno]{amsart}

\usepackage{mathrsfs}

\usepackage{color}
\usepackage{amsmath, amsthm, amssymb}
\usepackage{amsfonts}
\usepackage[ansinew]{inputenc}
\usepackage{graphicx}
\usepackage{caption}
\usepackage{subcaption}

\usepackage[english]{babel}
\usepackage{hyperref}

\usepackage{tikz}
\usepackage{rotating}

\usepackage{cite}
\usepackage{amscd}
\usepackage{color}
\usepackage{bm}
\usepackage{enumerate}

\usepackage{verbatim}
\usepackage{hyperref}
\usepackage{amstext}
\usepackage{latexsym}

\usepackage[pagewise]{lineno}
\theoremstyle{plain}
\newtheorem{theorem}{Theorem}[section]
\newtheorem{definition}[theorem]{Definition}

\newtheorem{proposition}[theorem]{Proposition}
\newtheorem{lemma}[theorem]{Lemma}

\numberwithin{theorem}{section}
\numberwithin{equation}{section}

\newcommand{\average}{{\mathchoice {\kern1ex\vcenter{\hrule height.4pt
width 6pt depth0pt} \kern-9.7pt} {\kern1ex\vcenter{\hrule
height.4pt width 4.3pt depth0pt} \kern-7pt} {} {} }}

\def\R{\mathbb{R}}

\def\div{\text{div}}

\renewcommand{\a }{\alpha }
\renewcommand{\b }{\beta }
\renewcommand{\d}{\delta }

\newcommand{\D }{\Delta }

\newcommand{\e }{\varepsilon }

\newcommand{\G }{\Gamma}

\newcommand{\n }{\nabla }

\renewcommand{\phi}{\varphi}

\renewcommand{\th }{\theta }

\renewcommand{\O }{\Omega }
\newcommand{\ov}{\overline}

\newcommand{\be}{\begin{equation}}
\newcommand{\ee}{\end{equation}}

\newcommand{\de}{\partial}

\newcommand{\Ds}{(-\D)^s}

\newcommand{\N}{\mathbb{N}}

\renewcommand{\H}{{\mathcal H}}

\newcommand{\cG}{{\mathcal G}}
\newcommand{\cH}{{\mathcal H}}

\newcommand{\cP}{{\mathcal P}}

\newcommand{\B}{{\bf B}}

\renewcommand{\epsilon}{\varepsilon}

 \renewcommand{\div}{\textrm{div}}

\newcommand{\pwkrm}{P_{r_m}^{x_m}W_{k_m}}
\newcommand{\Hplus}{\cH_{\O}^+}
\newcommand{\Hom}{\cH_+^{x_0,\nu}}
\newcommand{\Homk}{\cH_+^{x_k,\nu_k}}
\newcommand{\pd}[2]{\frac{\partial #1}{\partial #2}}

\newcommand{\bdis}{\begin{displaymath}}
\newcommand{\edis}{\end{displaymath}}
\newcommand{\ben}{\begin{eqnarray}}
\newcommand{\een}{\end{eqnarray}}
\newcommand{\bene}{\begin{eqnarray*}}
\newcommand{\eene}{\end{eqnarray*}}
\newcommand{\Rn}{\mathbb{R}^N}
\newcommand{\RN}{\R_+^{N+1}}

\newcommand{\cinfc}{C^\infty_c}

\begin{document}

\title[Regularity degenerate elliptic equation]
{Boundary Regularity for a degenerate elliptic equation with mixed boundary conditions}

\author[Alassane Niang ]{Alassane Niang}
\date{\today}
\address{\hbox{\parbox{5.7in}{\medskip\noindent{A. Niang\\
Department of Mathematics, Faculty of Science and Technics, Cheikh Anta Diop University of Dakar (UCAD)\\
Phone. : +221 33 825 05 30\\
B.P. 5005 Dakar-Fann, Senegal.\\[3pt]
         {\em{E-mail address: }}{\tt alassane4.niang@ucad.edu.sn.}}}}}

\begin{abstract}
	We consider a function $U$ satisfying a degenerate elliptic equation on $\RN:=(0,+\infty)\times\Rn$ with mixed Dirichlet-Neumann boundary conditions. The Neumann condition is prescribed on a bounded domain $\O\subset\Rn$ of class $C^{1,1}$, whereas the Dirichlet data is on the exterior of $\O$. We prove H\"older regularity estimates of $\frac{U}{d_\O^s}$, where $d_\O$ is a distance function defined as $d_\O(z):=\textrm{dist}(z,\Rn\setminus\O)$, for $z\in\ov{\RN}$. The degenerate elliptic equation arises from the Caffarelli-Silvestre extension of the Dirichlet problem for the fractional Laplacian. Our proof relies on compactness and blow-up analysis arguments.
\end{abstract}

\maketitle

\section{Introduction}\label{Intro}
	
	This paper is concerned with regularity estimates of solutions to degenerate mixed elliptic problems.	More precisely, for $s\in(0,1)$, we consider the differential operators $M_s$ and $N_s$ given by $M_sU(t,x):=\div_{t,x}(t^{1-2s}\nabla_{t,x} U)(t,x)$ and $N_sU(t,x):=-t^{1-2s}\pd{U}{t}(t,x)$, for $(t,x)\in(0,+\infty)\times\Rn$. Now let $f\in L^\infty(\Rn)$ and $U\in\dot{H}^1(t^{1-2s};\RN)$ satisfying
		\be\label{eq:pbm1}
		\left\lbrace
			\begin{array}{ll}
				M_sU=0 & \textrm{ in } \R^{N+1}_+, \\
				\lim_{t\rightarrow0}N_sU(t,\cdot)= f& \textrm{ on } \O,\\
				U=0 & \textrm{ on } \Rn\setminus\O,
			\end{array}
		\right.
		\ee
		where
		\bdis
			\dot{H}^1(t^{1-2s};\RN):=\left\lbrace W\in L^1_{loc}(\RN):~\int_{\RN}t^{1-2s}\vert\nabla W\vert^2dtdx<+\infty\right\rbrace.
		\edis
		Equation \eqref{eq:pbm1} is understood in the weak sense, see \eqref{eq:weak-sol}. Here and in the following, $\RN:=\{(t,x)\in\R\times\Rn:t>0\}$ and $\O$ is a bounded domain of class $C^{1,1}$ in $\Rn$.
Problem \eqref{eq:pbm1} is a weighted (singular or degenerate, depending on the value of $s\in(0,1)$) elliptic equation on $\RN$ with mixed boundary conditions. The weight $t^{1-2s}$ belongs to the Muckenhoupt class $A_2$, i.e. for any ball $\B\subset\R^{N+1}$, there exists a constant $C$ such that
		\bdis
			\left(\dfrac{1}{\vert \B\vert}\int_\B \vert t\vert^{1-2s}dtdx\right)\left(\dfrac{1}{\vert \B\vert}\int_\B \vert t\vert^{2s-1}dtdx\right)\leq C,
		\edis
		see \cite{Mucken-1972} for more details.\\
		Regularity estimates and Harnack inequalities for solutions to degenerate elliptic equations with mixed boundary conditions have been studied by many authors, we refer to \cite{FJK-wiener, FKS-degen, cabre-sire, Caff-Silv, JLX,zaremba-1910,lev-dee-93,bjorn-deg-ell-mixed-prob, kass-mady-ellipt-mixed, savare-regul-ellipt-prblm}. Important applications to these equations can be found in \cite{bucur-valdinoci-nonlocal-diffusions,mills-dudu-mbvp,fabrikant-mbvp}.

In the present paper, we are interested in the regularity of $\frac{U}{d_\O^s}$ up to the interface $\{0\}\times\de\O$ of Dirichlet and Neumann data. Equation \eqref{eq:pbm1} can be seen as a local version of the following fractional elliptic equation
	\be\label{eq:DEP}
	\left\lbrace
		\begin{array}{ll}
			(-\Delta)^su=f & \textrm{ in }\O,\\
			u=0 & \textrm{ in }\Rn\setminus\O,
		\end{array}
	\right.
	\ee
	where $u$ is the trace of $U$. Here $\Ds$ is the fractional Laplacian defined as $\Ds v(x)=C\lim_{\e\to 0}\int_{|x-y|>\e}\frac{v(x)-v(y)}{|x-y|^{N+2s}}\, dy$, with $C$ a positive normalization constant.
Indeed, in 2007, Caffarelli and Silvestre \cite{Caff-Silv} obtained an extension theorem that renders \eqref{eq:DEP} somewhat local. They proved that for every $u\in \dot{H}^s(\R^N)$, there exists a unique $U\in \dot{H}^{1}(t^{1-2s};\R^{N+1}_+)$ satisfying
	\bdis
		\left\lbrace
			\begin{array}{ll}
				M_sU=0 & \textrm{ in } \R^{N+1}_+, \\
				U=u & \textrm{ on } \Rn
			\end{array}
		\right.
	\edis
	and such that
		\bdis
			\lim_{t\rightarrow0}N_sU(t,\cdot)=k_s\Ds u,
		\edis
		where $k_s$ is a constant depending only on $s$, see e.g. \cite{cabre-sire}, and
		\bdis
			\dot{H}^s(\Rn):=\left\lbrace v\in L^1_{loc}(\Rn):~\int_{\Rn}\int_{\Rn}\dfrac{\left(v(x)-v(y)\right)^2}{\vert x-y\vert^{N+2s}}dxdy<+\infty\right\rbrace.
		\edis
		A function $U\in\dot{H}_{\O}(t^{1-2s};\RN)$ is a weak solution of \eqref{eq:pbm1} if
		\be\label{eq:weak-sol}
			\int_{\RN}t^{1-2s}\nabla U(t,x)\nabla\Psi(t,x)dtdx=k_s\int_{\O}f(x)tr(\Psi)(x)dx \qquad \textrm{ for all }\Psi\in\dot{H}_{\O}(t^{1-2s};\RN),
		\ee
		where $f$ is as in \eqref{eq:DEP}, $tr(\Psi)$ means trace of $\Psi$ on $\{0\}\times \Rn$ and
		$$
			\dot{H}_{\O}(t^{1-2s};\RN):=\left\lbrace U\in\dot{H}^1(t^{1-2s};\RN):tr(U)\in\dot{H}^s(\Rn)\right\rbrace.
		$$
		
		The Caffarelli-Silvestre extension, because of its local nature, is very often used to prove qualitative properties of solutions to problems involving the fractional Laplacian, see for instance \cite{FW-nonexistence,cabre-sire,silv-on-the-diff,JLX,kim-lee-jfa-2011,caff-salsa-silv}.
	Equation \eqref{eq:DEP} is a special case of integro-differential equations called nonlocal equations. The study of nonlocal equations have attracted several researchers in the last years since they appear in different physical models; from water waves, signal processing, materials sciences, financial mathematics etc. We refer to \cite{DPE-hitch} and the references therein for further motivations.\\
	Let us now recall some of the main boundary regularity results in the case of problem \eqref{eq:DEP} itself. In \cite{RS-bound}, Ros-Oton and Serra first proved that for $f\in L^\infty(\R^N)$ and $\O$ of class $C^{1,1}$, $u/\d^s_\O$ belongs to $C^{0,\a}(\ov\O)$, for some $\a\in (0,1)$ and $u$ satisfying \eqref{eq:DEP}.  Here and in the following  $\d_\O(x)=\textrm{dist}(x,\R^N\setminus\O)$.\\
Exploiting H\"ormander's theory for pseudo-differential operators, Grubb \cite{Grubb1,Grubb2} proved that $u/\d_\O^s\in C^\infty(\ov\O)$ if $f$ is $C^\infty-$regular and $\O$ of class $C^\infty$, for the fractional Laplacian. More recently, Ros-Oton and Serra \cite{RS-fully,RS-stable} extended and generalised their result to fully nonlinear nonlocal operators. They showed that if $f\in C^{0,\a}(\Rn)$ ($f\in L^\infty(\Rn)$) and $\O$ is of class $C^{2,\a}$ ($\O$ of class $C^{1,1}$) then $u/\d_\O^s\in C^{s+\a}(\ov{\O})$ ($u/\d_\O^s\in C^{s-\e}$ for any $\e>0$) for $\a>0$. Recently in \cite{fall-schroding-17}, the author proves H\"older estimates up to the boundary of $\O$, for $u$ and the ratio $\frac{u}{\d_\O^s}$, where $\O$ is of class $C^{1,\gamma}$, for $\gamma>0$ and $u$ is a weak solution of a nonlocal Schr\"odinger equation, with $f$ in some Morrey spaces.
		
		The main goal of this paper is to study the same type of regularity for problem \eqref{eq:pbm1}.  Our main result is stated in the following
		\begin{theorem}\label{t:I}
		Let $s\in(0,1)$, $f\in L^\infty(\Rn)$ and $\O$ be a bounded domain of class $C^{1,1}$ in $\Rn$. Let $W\in \dot{H}^{1}( t^{1-2s};\RN) $ be a weak solution to
		\be \label{eq:W-solves-introduc}
			\left\lbrace
			\begin{array}{ll}
				M_sW=0 & \textrm{ in } \R^{N+1}_+,\\
				\lim_{t\rightarrow0}N_sW(t,\cdot)=f & \textrm{ on } \O,\\
				W=0 &  \textrm{ on } \Rn\setminus\O.
			\end{array}
			\right.
		\ee
		Then, for any $0<\epsilon<s$, there exists a function $\Psi\in C^{s-\epsilon}([0,1]\times\overline{\O})$ such that
		\bdis
				W=d^s_\O\Psi.
		\edis
		Moreover,
		\bdis
			\left\lVert\Psi\right\rVert_{C^{s-\epsilon}([0,1]\times\overline{\O})}\leq C\left(\parallel W\parallel_{L^\infty(\RN)} + \parallel f\parallel_{L^\infty(\Rn)}\right),
		\edis
		where $d_\O(t,x)=(t^2+\d_\O^2(x))^{1/2}$, for $(t,x)\in\ov{\RN}$ and $\d_\O(x)=\textrm{dist}(x,\R^N\setminus\O)$. Here  $C$ is a positive constant depending only on $\O$,  $N$, $s$ and $\epsilon$.
	\end{theorem}
	The result in Theorem \ref{t:I} was known in the case $s=1/2$, $\O$ of class $C^\infty$ and $f\in C^\infty(\ov\O)$, see e.g. \cite{costabel-dauge-asymp-expans,grubb-mixed-2015}. It does not seem to be an immediate task to derive Theorem \ref{t:I} from the nonlocal result in \cite{RS-fully,RS-stable}, by e.g. the Poisson kernel representation. We therefore have to study in details \eqref{eq:W-solves-introduc}, although our argument is inspired by \cite{RS-fully}.
	
The proof of Theorem \ref{t:I} is inspired by \cite{RS-fully}, which we explain in the following. First, we let $h^+:\R^2\to \R$ be the (1-dimensional) solution  to
\begin{align}\label{eq:hplus}
\begin{cases}
M_sh^+(t,r)&=0\qquad\textrm{ for $r\in \R$ and $t>0$},\\
\lim_{t\rightarrow0}N_sh^+(t,r)&=0  \qquad\textrm{ for $r>0$},\\
h^+(0,r)&=0  \qquad\textrm{ for $r\leq 0 $ }.
\end{cases}
\end{align}
In particular $h^+(0,r)=\max(r,0)^s$, see \cite{RS-fully}. Let $\nu (x_0)$ be the unit interior normal to  $\de\O$ at $x_0$.  Given $x_0\in \de \O$,  the function $\cH_+^{x_0,\nu}(t,x)=h^+(t, (x-x_0)\cdot \nu(x_0)) $ satisfies \eqref{eq:pbm1}, for $f=0$. For an explicit expression of  $h^+$, see Section  \ref{Sect3}. We note that $\cH_+^{x_0,\nu}$ belongs to the space
	\bdis
		L^2(t^{1-2s};\B):=\left\lbrace V:\RN\to\R,~\int_{\B}t^{1-2s}\vert V\vert^2dtdx<+\infty \right\rbrace,
	\edis
	for an open set $\B\subset\subset\ov\RN$.\\
	The main goal is then to derive the estimate
\be\label{eq:Taylor-expand-near-de-Om} 
| W(z)- Q(x_0)\Hom(z)|\leq CC_0 |z-x_0|^{2s-\e}, \qquad\textrm{ for all $z\in [0,1)\times B_{1/2}$},
\ee
	where $Q(x_0)\in\R$, $C_0=\|W\|_{L^\infty(\RN)}+\|f\|_{L^\infty(\R^N)}$ and $C$ is a positive constant depending only on $N$, $s,~\e$ and $\O$. Moreover $|Q(x_0)\Hom(z)|\leq C$ for every $x_0\in \de\O\cap B_{1/2}$ and $z\in\B_1^+(x_0)$.  We note that \eqref{eq:Taylor-expand-near-de-Om} can be seen as a Taylor expansion of $W$ near the interface $\{0\}\times\de\O$. To reach \eqref{eq:Taylor-expand-near-de-Om}, we use blow up analysis combined with a regularity estimate on $\ov\RN$ and the Liouville-type result on the half-space contained in Lemma \ref{l:Vmholder}. This argument was developped by Serra \cite{Serra-inter-reg} to prove interior regularity results for fully nonlinear nonlocal parabolic equations and by Ros-Oton and Serra \cite{RS-fully} to prove boundary regularity estimates for integro-differential equations.  Once we get \eqref{eq:Taylor-expand-near-de-Om}, we now deduce the result in the main theorem.

	The paper is organized as follows. In Section \ref{Sect2}, we give some notations and definitions of functional spaces and their associated norms for the need of this work. We state some preliminaries in Section \ref{Sect3}. In Section \ref{Sect4}, we prove an intermediate boundary regularity result for solution to equation \eqref{eq:pbm1} on $\ov\RN$ with the Neumann boundary condition only. We use blow up analysis and compactness arguments to prove \eqref{eq:Taylor-expand-near-de-Om} in Section \ref{Sect5}. In Section \ref{Sect6}, we prove some regularity estimates in the neighbourhood of the interface set $\de\O$. In Section \ref{Section7}, we give the complete proof of Theorem \ref{t:I}.\\
\textbf{Acknowledgement:} The author would like to thank Diaraf Seck and Mouhamed Moustapha Fall for helpfull discussions and encouragements. This work is supported by the NLAGA Project of the Simons foundation and the Post-AIMS bursary of AIMS-SENEGAL.\\

\section{Definitions and Notations}\label{Sect2}

	We start by introducing some spaces and their norms. Let $s\in(0,1)$, we define
	\bdis
	H^s(\Rn):=\left\lbrace v\in L^2(\Rn):~\int_{\Rn}\int_{\Rn}\dfrac{\left(v(x)-v(y)\right)^2}{\vert x-y\vert^{N+2s}}dxdy<+\infty\right\rbrace.
	\edis
	This space is endowed with the norm
	\bdis
		\Vert v\Vert_{H^s(\Rn)}:=\left(\int_{\Rn}\vert v\vert^2dx+\int_{\Rn}\int_{\Rn}\dfrac{\left(v(x)-v(y)\right)^2}{\vert x-y\vert^{N+2s}}dxdy\right)^{1/2}.
	\edis
	We let
	\be\label{eq:l-very-weak}
		\mathcal{L}_s(\Rn):=\left\lbrace v\in L^1_{loc}(\Rn):~ \int_{\Rn}\dfrac{\vert v(x)\vert}{1+\vert x\vert^{N+2s}}dx<+\infty\right\rbrace.
	\ee
	For $a\in(-1,1)$ and an open set $\B\subset\subset\ov\RN$, we denote
	\bdis
		L^2(t^a;\B):=\left\lbrace V:\RN\to\R,~\int_{\B}t^a\vert V\vert^2dtdx<+\infty \right\rbrace
	\edis
	endowed with the norm
	\bdis
		\Vert V\Vert_{L^2(t^a;\B)}:=\left(\int_{\B}t^a\vert V\vert^2dtdx\right)^{1/2}
	\edis
	and
	\bdis
		H^1(t^a;\B):=\left\lbrace V\in L^2(t^a;\B):\n V\in L^2(t^a;\B)\right\rbrace,
	\edis
	with the induced  norm
	\bdis
		\Vert V\Vert_{H^1(t^a;\B)}:=\left(\int_{\B}t^a\left(\vert V\vert^2+\vert\n V\vert^2\right)dtdx\right)^{1/2}.
	\edis

We recall the fractional Laplacian of $u\in\mathcal{L}_s(\Rn)\cap C^2_{loc}(\Rn)$,
	\be\label{eq:frlp}
		\Ds u(x):=C_{N,s}P.V.\int_{\Rn}\dfrac{u(x)-u(y)}{\vert x-y\vert^{N+2s}}dy,
	\ee
	where $C_{N,s}= \pi^{2s+N/2}\dfrac{\Gamma(s+N/2)}{\Gamma(-s)}$, $\Gamma$ is the usual Gamma function and P.V. is the Cauchy Principal Value.\\
	For $s\in(0,1)$, $\dot{H}^s(\Rn)$ coincides with the trace of $\dot{H}^1(t^{1-2s};\RN)$ on $\partial\RN=\{(t,x)\in\R\times\Rn:t=0\}$. In particular, every function $U\in \dot{H}^1(t^{1-2s};\RN)$ has a unique trace function $u=U_{\vert\Rn}\in \dot{H}^s(\Rn)$, see \cite{cabre-sire}.
	
	Let $f$ be a function and $\a\in(0,1)$, the H\"older seminorm of $f$ is given by
	\bdis
		\left[f\right]_{C^{0,\alpha}(\O)}:=\sup_{\substack{x,y\in\O \\ x\neq y}}\dfrac{\vert f(x)-f(y)\vert}{\vert x-y\vert^\alpha}.
	\edis
	For $k\in\N$, $f\in C^{k,\alpha}(\O)$ means that the quantity
	\bdis
		\Vert f\Vert_{C^{k,\alpha}(\O)}:=\sum_{l=0}^k\sup_{x\in \O}\left\lvert D^lf(x)\right\rvert + \left[D^kf\right]_{C^{0,\alpha}(\O)}
	\edis
	is finite. In this work, instead of writing $C^{k,\alpha}$, we will put $C^{k+\alpha}$ sometimes for the same definition.
	
	Let us now introduce some notations used throughout the paper, 
		\be\label{eq:ball}
			B_R(x_0):=\{x\in\Rn : \vert x-x_0\vert<R\},\qquad \B_R^+(x_0):=[0,R)\times B_R(x_0)
		\ee
		and
		$$
			\B_R(z_0):=\{z=(t,x)\in\R\times\Rn : \vert z-z_0\vert<R\}
		$$
		is the ball of center $z_0=(t_0,x_0)\in\R^{N+1}$ and radius $R$. We will use the variables $x$ and $z$ for the spaces $\Rn$ and $\R^{N+1}$ respectively. For simplicity, when $x_0=0$ and $z_0=0$, we simply write $B_R$ (or $\B_R^+$) and $\B_R$ respectively.
		
	We also define the distance functions $\d$ and $d$ by
		$$
			\delta(x):=\textrm{dist}(x,\R^N\setminus \O) \quad \textrm{and}\quad d(t,x):=\left(t^2+\delta^2(x)\right)^{1/2}=\textrm{dist}\left(z,\Rn\setminus\O\right),
		$$
		where $z=(t,x)\in\ov\RN$.
		
		Finally, for $x_0\in \de\O$, we let $\nu(x_0)$ be the interior normal  to $\de\O$ at $x_0$. We then  define 
		$$
			\bar{\d}_{x_0,\nu}(x):=\left[(x-x_0)\cdot\nu(x_0)\right]_+,~x\in\O.
		$$
		
		In this paper, all constants $C$ or $C(N,s)$ that we do not specify are positive universal constants.
\section{Preliminaries}\label{Sect3}

Let $\cH_+(t,x)= {h^+}(t,x_N),\forall x\in\Rn$ and $t>0$, where $h^+$ is as in \eqref{eq:hplus}. Then we have that   
\begin{align*}
	\begin{cases}
		M_s\cH_+=0 & \textrm{ in } \R^{N+1}_+,\\
		\lim_{t\rightarrow0}N_s\cH_+(t,\cdot)=0 & \textrm{ on } \{x_N>0\},\\
		\cH_+= 0&  \textrm{ on } \{x_N\leq 0\}.
	\end{cases}
\end{align*}
Recall that $\cH_+(t,x)=\cP(t,\cdot)\star(x_N)_+^s$, where $\cP(t,x)=C(N,s)\frac{t^{2s}}{\left(t^2+\vert x\vert^2\right)^{\frac{N+2s}{2}}}$ and $\star$ denotes the convolution product. For every $\d\in\R$ and $t>0$, we 	let
	$$
	h^+(t,\d)=C_{N,s}'t^{2s}\int_{\R}\dfrac{(y_N)_+^s}{(t^2+|\d-y_N|^2)^\frac{1+2s}{2}}dy_N= C_{N,s}'\int_{\R}\dfrac{(\d+t\rho)_+^s}{(1+\rho^2)^\frac{1+2s}{2}}\, d\rho.
	$$
 See for instance \cite{RS-fully}, using polar coordinates, letting $t=r \sin\th$ and $\d=r \cos\th$, with $\th\in (0,\pi)$ and $r>0$, we have 
\bdis
	h^+(t,\d)= C r^s\cos^{2s}(\th/2)\,  {}_2F_1\left(0, 1;1-s; \frac{1-\cos\th}{2}\right),
\edis
where $r=\sqrt{t^2+\d^2}$, $\th=\textrm{arctan}(\frac{t}{\d})$. Here $ {}_2F_1 $ is  the Hypergeometric function which can be expressed by the power series, for $0<x<1$,
\be\label{eq:power-series}
{}_2F_1\left(0, 1;1-s; x\right)=\sum_{n=0}^\infty a_n \frac{x^n}{n!},
\ee
with $a_n>0$.\\
Next, we consider a bounded domain $\O$ of class $C^{1,1}$. We denote by $\nu$ the interior normal to $\de\O$.  For $x_0\in\de\O$, we will consider the  function
\bdis
	\cH_+^{x_0,\nu}(t,x)=h^+(t, (x-x_0)\cdot \nu(x_0)),~x\in\O.
\edis
It is clear that
\begin{align}\label{eq:cH-solves}
\begin{cases}
	M_s\Hom=0 & \qquad \textrm{ in } \R^{N+1}_+,\\
	\lim_{t\rightarrow0}N_s\Hom(t,\cdot)=0 & \qquad \textrm{ on } \{(x-x_0)\cdot \nu(x_0)>0\},\\
	\Hom= 0 &  \qquad \textrm{ on } \{(x-x_0)\cdot \nu(x_0)\leq 0\}.
\end{cases}
\end{align}

\section{Regularity estimate up to the boundary for the degenerate equation with the Neumann boundary condition}\label{Sect4}

	The following result is stronger than needed since the $C^{s-\e}$ estimate for the solution $V$ in $\ov{\B_1^+}$ will be enough for our purpose.
	\begin{theorem}\label{t:interegul}
		Let $s\in(0,1)$ and $f\in L^\infty(\Rn)$. Let $V\in L^\infty(\RN)\cap H^1(t^{1-2s};\B_2^+)$ satisfy
		\begin{align*}
		\begin{cases}
			M_sV= 0 & \qquad \textrm{ in } \B_2^+,\\
			\lim_{t\rightarrow0}N_sV(t,\cdot)=f & \qquad \textrm{ on } B_2.
		\end{cases}
		\end{align*}
		Then $V\in C^{2s-\epsilon}(\ov{\B_1^+})$ for all $0<\epsilon<2s$. Moreover,
		\bdis
			\parallel V\parallel_{C^{2s-\epsilon}(\ov{\B_1^+})}\leq C\left(\parallel V\parallel_{L^\infty(\RN)}+\parallel f\parallel_{L^\infty(\Rn)}\right),
		\edis
		where $C$ is a positive constant depending only on $N$, $s$ and $\epsilon$.
	\end{theorem}
	\begin{proof}
		Consider the cut-off function $\eta\in\cinfc(B_3)$ such that $\eta\equiv 1$ in $B_2$ and $0\leq \eta\leq 1$ in $\Rn$. Let $\ov{v}$ be the (unique) solution to the equation
		\bdis
			\Ds \ov{v}=\ov f \qquad \textrm{ in } \Rn,
		\edis
		where $\ov f:=\eta f$. By \cite[Proposition 2.19]{silv-obstac}, $\ov{v}\in C^{2s-\e}(\Rn)$ and
		\bdis
			\parallel\ov{v}\parallel_{C^{2s-\e}(\Rn)}\leq C\left(\parallel\ov{v}\parallel_{L^\infty(\Rn)}+\parallel \ov{f}\parallel_{L^\infty(\Rn)}\right),
		\edis
		where $C>0$ is a constant that depends only on $N$, $s$ and $\epsilon$.
		Now consider the Caffarelli-Silvestre extension $\ov{V}$ of $\ov{v}$, i.e
		\bdis
			\ov{V}(t,\cdot)=\cP(t,\cdot)\star\ov{v}
		\edis
		that verifies the equation
		\bdis
			\left\lbrace
			\begin{array}{ll}
				M_s\ov{V}=0 & \text{ in } \RN,\\
				\lim_{t\rightarrow0}N_s\ov{V}(t,x)=\Ds\ov{v}(x)=\ov{f}(x) & \text{ on } \Rn.
			\end{array}
			\right.
		\edis
		By a change of variable, we have
		\be\label{eq:conv-poisson-kernel}
			\ov{V}(t,x)=\left(\cP(t,\cdot)\star\ov{v}\right)(x)=\int_{\Rn}\ov{v}(x-ty)H_s(y)dy,
		\ee
		where $H_s(y)=\cP(1,y)=\dfrac{C}{(1+\vert y\vert^2)^{\frac{N+2s}{2}}}$ and verifies $\int_{\Rn}H_s(y)dy=1$. Then, for $z_1=(t_1,x_1),z_2=(t_2,x_2)\in\ov\RN$, we have
		\bene
			\vert \ov{V}(z_1)-\ov{V}(z_2)\vert & \leq & \vert z_1-z_2\vert^{2s-\epsilon}\parallel\ov{v}\parallel_{C^{2s-\epsilon}(\Rn)}\int_{\Rn}\max\{\vert y\vert^{2s-\epsilon},1\}H_s(y)dy,\\
			& \leq & C\vert z_1-z_2\vert^{2s-\epsilon}\left(\parallel\ov{v}\parallel_{L^\infty(\Rn)}+\parallel \ov{f}\parallel_{L^\infty(\Rn)}\right).
		\eene
		By \eqref{eq:conv-poisson-kernel}, it is clear that
		\bdis
			\parallel\ov{V}\parallel_{L^\infty(\RN)}\leq \parallel \ov{v}\parallel_{L^\infty(\Rn)} \leq C\parallel \ov{f}\parallel_{L^\infty(\Rn)}.
		\edis
		Therefore
		\be\label{eq:Vbarreg}
			\parallel\ov{V}\parallel_{C^{2s-\epsilon}(\ov\RN)}\leq C\left(\parallel\ov{V}\parallel_{L^\infty(\RN)}+\parallel \ov{f}\parallel_{L^\infty(\Rn)}\right),
		\ee
		where the constant $C>0$ depends only on $N$, $s$ and $\epsilon$.\\
		Now put $\widetilde{V}=V-\ov{V}$, then $\widetilde{V}$ satisfies
		\bdis
			\left\lbrace
			\begin{array}{ll}
				M_s\widetilde{V}=0 & \text{ in } \B_2^+,\\
				\lim_{t\rightarrow0}N_s\widetilde{V}(t,\cdot)=(1-\eta) f=0 & \text{ on } B_2.
			\end{array}
			\right.
		\edis
		Considering the even reflexion $\widetilde{W}$ of $\widetilde{V}$ in the variable $t$, we have
		\bdis
			\div\left(\vert t\vert^{1-2s}\nabla\widetilde{W}\right)=0 \textrm{ in } \B_2.
		\edis
		From \cite[Corollary 2.5]{caff-salsa-silv}, we have that for $x\in B_1$ and $t\in(-1,1)$ fixed,
		\be\label{eq:Dx2}
			 \left\lvert D^2_x\widetilde{W}(t,x)\right\rvert \leq C\parallel\widetilde{W}(t,\cdot)\parallel_{L^{\infty}(B_2)}.
		\ee
		By \cite[Proposition 2.6]{caff-salsa-silv}, we obtain
		\bdis
			\left\lvert \widetilde{W}_{tt}+\frac{1-2s}{\vert t\vert}\widetilde{W}_t\right\rvert\leq C\parallel \widetilde{W}\parallel_{L^\infty(\B_2)}.
		\edis
		Therefore
		\bdis
		 	\left\lvert \left(\vert t\vert^{1-2s}\widetilde{W}_t\right)_t\right\rvert \leq C \vert t\vert^{1-2s}\parallel \widetilde{W}\parallel_{L^{\infty}(\B_2)}
		\edis
		and hence
		\bdis
		 	\left\lvert \widetilde{W}_{tt}\right\rvert \leq C\parallel\widetilde{W}\parallel_{L^{\infty}(\B_2)}.
		\edis
		For $(t,x)\in\ov{\B_1}$, we have, by \eqref{eq:Dx2}, that
		\bdis
			\left\lvert \widetilde{W}_{tt}(t,x)\right\rvert + \left\lvert D^2_x\widetilde{W}(t,x)\right\rvert \leq C\parallel\widetilde{W}\parallel_{L^{\infty}(\B_2)}
		\edis
		which implies that
		\bdis
			\widetilde{W}\in C^{2-\epsilon}(\ov{\B}_1).
		\edis
		Thus, it follows that $ \widetilde{V}\in C^{2-\epsilon}(\ov{\B_1^+})$ and
		\bene
			\parallel \widetilde{V}\parallel_{C^{2-\epsilon}(\ov{\B_1^+})} & \leq & C\left(\parallel V\parallel_{L^{\infty}(\B_2^+)}+\parallel\ov{v}\parallel_{L^{\infty}(\Rn)}\right),\\
			& \leq & C\left(\parallel V\parallel_{L^{\infty}(\B_2^+)}+\parallel f\parallel_{L^{\infty}(B_2)}\right).
		\eene
		We finally obtain
		\bene
			\parallel V\parallel_{C^{2s-\epsilon}(\ov{\B_1^+})} & \leq & C\left( \parallel \widetilde{V}\parallel_{C^{2-\epsilon}(\ov{\B_1^+})}+\parallel \ov{V}\parallel_{C^{2s-\epsilon}(\ov{\RN})}\right),\\
			& \leq & C\left(\parallel V\parallel_{L^{\infty}(\B_2^+)}+\parallel f\parallel_{L^{\infty}(B_2)}\right),
		\eene
		since $\widetilde{V}=V-\ov{V}$.
	\end{proof}
	
\section{Toward regularity by blow up analysis}\label{Sect5}
\noindent
		For local boundary regularity results in $C^{1,1}$ domains, we fix the geometry of the domain as follows:
	\begin{definition}\label{d:fix-domain}
		We define $\cG$ the set of all interfaces $\Gamma$ with the following properties:\\
		there are two disjoint domains $\O^+$ and $\O^-$ satisfying $\ov{B_1}=\ov{\O^+}\cup\ov{\O^-}$ such that
		\begin{itemize}
			\item $\Gamma:=\de\O^+\setminus\de B_1=\de\O^-\setminus\de B_1$;
			\item $\G$ is a $C^{1,1}$ hypersurface;
			\item  $0\in\Gamma$.
		\end{itemize}
	\end{definition}
	For $\Gamma\in\mathcal{G}$, we let $x_0 \in \Gamma\cap B_{1/2}$ and $W,~\Hom\in L^2(t^{1-2s};\B_r^+(x_0))$, for all $r>0$. Consider the $1$-dimensional subspace of $L^2(t^{1-2s};\B_r^+(x_0))$ spanned by $\Hom$ and given by
	$$
		\R\Hom=\left\lbrace Q\Hom, Q\in\R\right\rbrace\subset L^2(t^{1-2s};\B_r^+(x_0)).
	$$
	Let $P_r^{x_0}W$ be the orthogonal $L^2(t^{1-2s};\B_r^+(x_0))$-projection of $W$ on $\R\Hom$, that is
	\bdis
		\min_{h\in\R\Hom}\Vert W-h\Vert_{L^2(t^{1-2s};\B_r^+(x_0))}=\Vert W-P_r^{x_0}W\Vert_{L^2(t^{1-2s};\B_r^+(x_0))},
	\edis
	then $P_r^{x_0}W=Q_r(x_0)\Hom$, where
	$$
		Q_r(x_0)=\dfrac{\int_{\B_r^+(x_0)}W(z)\Hom(z)dz}{\int_{\B_r^+(x_0)}\left(\Hom(z)\right)^2dz}\in\R.
	$$
	Moreover $P_r^{x_0}W$ has the property that
	\be\label{eq:property1}
		\int_{\B_r^+(x_0)}\left(W-P_r^{x_0}W\right)(z)P^{x_0}_rW(z)dz=0.
	\ee
	We now state the following lemma which will be  useful later.	
	\begin{lemma}\label{l:2}
 		Let $s\in(0,1)$ and $W\in H^1(t^{1-2s};\B_1^+)$. Let $P^{x_0}_rW$ be the orthogonal $L^2(t^{1-2s};\B_r^+(x_0))$-projection of $W$ on $\R\Hom$ and suppose that for all $r\in(0,1)$,
		\bdis
			\Vert W-P^{x_0}_rW\Vert_{L^\infty(\B_r^+(x_0))}\leq C_0r^{2s-\e}.
		\edis
		Then, there exists $Q(x_0)\in\R$ with $ |Q(x_0)|\leq C$ such that, letting
	\be\label{eq:px0w}
		P^{x_0}W=Q(x_0)\Hom,
	\ee
	we have
		\bdis
			\parallel W-P^{x_0}W \parallel_{L^\infty(\B_r^+(x_0))}\leq CC_0r^{2s-\e},
		\edis
		where the constant $C>0$ depends only on $N$, $s$ and $\e$.
	\end{lemma}
	\begin{proof}
		The proof is similar to the one of \cite[Lemma 6.2]{RS-fully}. We skip the details.
 	\end{proof} 	
	The main result of this section is contained in the following	
	\begin{proposition}\label{p:pbm}
		Let $\Gamma\in\cG$, see Definition \ref{d:fix-domain}.
		We let $s\in(0,1)$, $f\in L^\infty(\Rn)$ and assume that $W\in \dot{H}^1(t^{1-2s},\RN)\cap L^\infty(\RN)$  satisfies
		\be\label{eq:W}
			\left\lbrace
			\begin{array}{ll}
				M_sW=0 & \text{ in } \RN,\\
				\lim_{t\rightarrow0}N_sW(t,\cdot)=f & \text{ on } \O^+, \\
				W=0 &  \text{ on } \O^-.
			\end{array}
			\right.
		\ee
		Then, for all $x_0 \in \Gamma \cap B_{1/2}$,
		\be\label{eq:estimblowup}
			\vert W(z)-P^{x_0}W(z)\vert\leq C\vert z-x_0\vert^{2s-\epsilon}\left(\Vert W\Vert_{L^\infty(\RN)}+\Vert f\Vert_{L^\infty(\Rn)}\right), \quad\textrm{ for all } z\in   \B_1^+,
		\ee
		where $P^{x_0}W$ is given by \eqref{eq:px0w} and the positive constant $C$ depends only on $N,~s,~\e$ and $\Gamma$.
	\end{proposition}
	Remark that we can replace $\RN$ by $\B_1^+$ in \eqref{eq:W}.
	\begin{proof}
		For any $k\geq 1$, let $(\Gamma_k)\subset\mathcal{G},~(W_k)\subset\dot{H}^1(t^{1-2s},\RN)\cap L^\infty(\RN)$ and $(f_k)\subset L^\infty(\Rn)$ be sequences such that $W_k$ satisfies \eqref{eq:W} on $\RN$, the Neumann data on $\O_k^+$ is $f_k$ and the Dirichlet condition is on $\O_k^-$. Let $P_r^{x_k}W_k$ be the orthogonal $L^2(t^{1-2s};\B_r^+(x_k))-$projection of $W_k$ on $\R\Homk$ and $\nu_k\to\nu\in S^{N-1}$ the normal vector to $\Gamma_k$ towards $\O^+_k$. We suppose that $\lVert W_k\rVert_{L^\infty(\RN)}+\lVert f_k\rVert_{L^\infty(\Rn)}\leq 1$, for any $k\geq 1$.\\
		Assume that \eqref{eq:estimblowup} is not true, then by Lemma \ref{l:2},
\noindent
		\bdis
			\sup_{k\geq 1}\sup_{r>0}~r^{-2s+\e}~\lVert W_k-P_r^{x_k}W_k\rVert_{L^\infty(\B_{r}^+(x_k))}=+\infty.
		\edis
		Set
		\bdis
			\Theta(r):=\sup_{k}\sup_{r'>r}\dfrac{\parallel W_k-P_{r'}^{x_k}W_k\parallel_{L^\infty(\B_{r'}^+(x_0))}}{(r')^{2s-\e}}.
		\edis
		Clearly, $\Theta$ is a monotone nonincreasing function, it verifies
		\bdis
			\Theta(r)\nearrow+\infty\qquad \textrm{as}\quad r\searrow 0
		\edis
		and $\Theta(r)<+\infty$ for $r>0$, because $\Vert W_k\Vert_{L^\infty(\RN)}\leq 1$. Thus, by definition of the supremum, there exist sequences $r_m\searrow 0,~k_m$ and $x_m\rightarrow x_0\in\Gamma\cap B_{1/2}$ such that
		\be\label{eq:theta2}
			\dfrac{\parallel W_{k_m}-\pwkrm\parallel_{L^\infty(\B_{r_m}^+)(x_m)}}{r_m^{2s-\e}\Theta(r_m)}\geq\frac{1}{2}.
		\ee
		Let us consider the sequence
		\be\label{eq:blowup-seq}
			V_m(z):=\frac{W_{k_m}(x_m+r_mz)-\pwkrm(x_m+r_mz)}{r_m^{2s-\e}\Theta(r_m)}.
		\ee
		Then by \eqref{eq:theta2}, we get
		\be\label{eq:theta3}
			\parallel V_m\parallel_{L^\infty(\B_{1}^+)}\geq \frac{1}{2}.
		\ee
		Also by \eqref{eq:property1}, we obtain the orthogonality condition
		\be\label{eq:optimal-condition}
			\int_{\B_{1}^+}V_m(z)\cH_+^{x_m,\nu_m}(z) dz=0.
		\ee		
		Now, let $R\geq 1$ be fixed, $m$ large enough so that $r_mR<\frac{1}{2}$ and $z\in \B_{2R}^+(x_m)$, we have that
		\bene
				M_sV_m(z) &=&\frac{r_m^{2s-(2s-\e)}}{\theta(r_m)}\div\left((r_mt)^{1-2s}\nabla \left(W_{k_m}-\pwkrm\right)\right)(x_m+r_mz),\\
				&=& 0,
		\eene
		where we used \eqref{eq:cH-solves} and \eqref{eq:W}. We have also that
		\bene
			\lim_{t\rightarrow0}N_sV_m(t,x) &=&-\frac{r_m^{\e}}{\theta(r_m)}\lim_{r_mt\to 0}(r_mt)^{1-2s}\pd{(W_{k_m}-\pwkrm)}{t}(r_mt,x_m+r_mx),\\
			&=&\frac{r_m^{\e}}{\theta(r_m)}f(x_m+r_mx),\quad x\in\O_m^\star\cap B_{2R}(x_m),
		\eene
		where $\O_m^\star:=\{x\in\Rn : x_m+r_mx\in\O_{k_m}^+ \textrm{ and } (x-x_m)\cdot\nu_m(x_m)>0\}$. Then $V_m$ satisfies
		\begin{align*}
			\begin{cases}
				M_sV_m(t,x) =0 & (t,x)\in\B_{2R}^+(x_m),\\
				\lim_{t\rightarrow0}N_sV_m(t,x)=\frac{r_m^{\e}}{\theta(r_m)}f(x_m+r_mx) & x\in\O_m^\star\cap B_{2R}(x_m),\\
				V_m(0,x)=0 & x\in(\O_m^\star)^c\cap B_{2R}(x_m).
			\end{cases}
		\end{align*}
		By Lemma \ref{l:Vmholder}, see below, up to a subsequence,
		$$
			V_m\to V_\infty\in\R\Hom \textrm{ uniformly on compact subsets of $\ov\RN$, as $m\to +\infty$}
		$$
		and further $V_\infty$ satisfies
		\begin{align*}
			\begin{cases}
				M_sV_\infty=0 & \textrm{ in } \RN,\\
				\lim_{t\rightarrow0}N_sV_\infty(t,\cdot)=0 & \textrm{ on }\{(x-x_0)\cdot\nu(x_0)>0\},\\
				V_\infty=0 & \textrm{ on } \{(x-x_0)\cdot\nu(x_0)\leq 0\}.
			\end{cases}
		\end{align*}
		Passing to the limit in \eqref{eq:theta3} and \eqref{eq:optimal-condition}, we get a contradiction.
	\end{proof}
	
	The following result was used in the proof of Proposition \ref{p:pbm}.
	\begin{lemma}\label{l:Vmholder}
		Let $V_m$ be the same sequence given by \eqref{eq:blowup-seq} in the proof of Proposition \ref{p:pbm}. Then, up to a subsequence,
		\bdis
			V_m\to K\Hom:=K\cP(t,\cdot)\star\bar{\delta}_{x_0,\nu}^s,\textrm{ as $m\to+\infty$}
		\edis
		uniformly on compact subsets of $\overline{\RN}$, where $K\in\R$.
	\end{lemma}
	\begin{proof}
		For $m$ fixed, consider the function $v_m$, the trace of the function $V_m$ such that
	\bdis
		\left\lbrace
		\begin{array}{ll}
			M_sV_m(t,x)=0 &  (t,x)\in\R^{N+1}_+,\\
			\lim_{t\rightarrow0}N_sV_m(t,x)=f_m(x) & x\in\O_m^\star,\\
			V_m(0,x)= 0 & x\in\Rn\setminus\O_m^\star,\\
			V_m(0,x)=v_m(x) & x\in\Rn,
		\end{array}
		\right.
	\edis
	with $\O_m^\star:=\{x\in\Rn : x_m+r_mx\in\O_{k_m}^+ \textrm{ and } (x-x_m)\cdot\nu_m(x_m)>0\}$, where $\nu_m(x_m)$ is a unit normal vector to $\Gamma_{k_m}$ at $x_m$ pointing towards $\O_{k_m}^+$ and $f_m\to 0$ as $m\to+\infty$. In particular, from the first and the last equation above, we have that 
	\bdis
		V_m(t,x):=\int_{\Rn}v_m(y)\cP(t,x-y)dy=C\int_{\Rn}\dfrac{v_m(x-ty)}{(1+\vert y\vert^2)^{\frac{N+2s}{2}}}dy.
	\edis
	We notice that from \cite[Proof of Proposition 8.3]{RS-fully} (for $\a=0$), the sequence $v_m$ satisfies the following estimates,
	\be\label{eq:vmbound}
		\parallel v_m\parallel_{L^\infty(B_R)}\leq R^\beta,~ \textrm{ for every $R>1$ and $0<\b<2s$}
	\ee
	and for every $m$ such that $r_mR\leq 1$,
	\be\label{eq:vmholder}
		\parallel v_m\parallel_{C^{0,\alpha}(B_R)}\leq C(R) \textrm{ for some } \a\in(0,1).
	\ee
	Recalling $\bar{\delta}_{x_0,\nu}=\left[(x-x_0)\cdot\nu(x_0)\right]_+$, we have
	$$
		v_m\to K\bar{\delta}_{x_0,\nu}^s \textrm{ in $C^{0,\alpha}$ on compact subsets of $\Rn$ for some }K\in\R,~x_0\in\partial\RN \textrm{ and } \nu\in S^{N-1}.
	$$
	First, let us prove that
		\bdis
			\parallel V_m\parallel_{L^{\infty}(\overline{\B_R^+})}\leq CR^\beta, \textrm{ for every } R>1
		\edis
		and that for every $m\in\N$ such that $r_mR\leq 1$,
		\bdis
			\parallel V_m\parallel_{C^{0,\alpha}(\overline{\B_R^+})}\leq C(R),
		\edis
		where $C(R)$ depends on $R$. For $R>1$, we consider the cut-off function $\eta_R\in\cinfc(B_{3R})$ such that $\eta_R\equiv 1$ on $B_{2R}$ and $\vert \eta_R\vert<1$ on $\Rn$. Then, we can write
		\bdis
			V_m(t,x)=\int_{\Rn}(\eta_Rv_m)(x-ty)H(y)dy + \int_{\Rn}\left((1-\eta_R)v_m\right)(x-ty)H(y)dy,
		\edis
		where $H(y)=\frac{C}{\left(1+\vert y\vert^2\right)^{\frac{N+2s}{2}}}$.
		We set
		\bdis
			V^1_m(t,x):=\int_{\Rn}(\eta_Rv_m)(x-ty)H(y)dy.
		\edis
		For every $R>1$ and $(t,x)\in\overline{\B_R^+}$,
		\ben\label{eq:vm1bound}
			\vert V_m^1(t,x)\vert & \leq & C\int_{\Rn}\dfrac{\left\lvert (\eta_Rv_m)(x-ty)\right\rvert}{(1+\vert y\vert^2)^{\frac{N+2s}{2}}}dy,\nonumber\\
			&\leq &  C\parallel \eta_Rv_m\parallel_{L^{\infty}(\Rn)}\int_{\Rn}\frac{dy}{(1+\vert y\vert^2)^{\frac{N+2s}{2}}},\nonumber\\
			& \leq & C\parallel v_m\parallel_{L^{\infty}(B_{3R})}\leq CR^\beta,
		\een
		by using \eqref{eq:vmbound}. Now, for every $m\in\N$ such that $r_mR\leq 1$ and $z_1,z_2\in \overline{\B_R^+}$, we get
		\bene
			\left\lvert V_m^1(z_1)-V_m^1(z_2)\right\rvert & \leq & C\int_{\Rn}\dfrac{\left\lvert (\eta_Rv_m)(x_1-t_1y)-(\eta_Rv_m)(x_2-t_2y)\right\rvert}{(1+\vert y\vert^2)^{\frac{N+2s}{2}}}dy,\\
			& \leq & C\vert z_1-z_2\vert^\a[ \eta_Rv_m]_{C^{0,\alpha}(\Rn)}\int_{B_{3R}}\frac{\max(1,|y|^\alpha)}{(1+\vert y\vert^2)^{\frac{N+2s}{2}}}dy,\\
			& \leq & C(R)\vert z_1-z_2\vert^\alpha,
		\eene
		where we have used \eqref{eq:vmholder} in the last inequality. Thus, for every $m\in\N$ such that $r_mR\leq 1$, we have
		\be\label{eq:vm1holder}
			\parallel V^1_m\parallel_{C^{0,\alpha}(\overline{\B_R^+})}\leq C(R).
		\ee
		Next, we define
		\bdis
			V^2_m(t,x):=\int_{\Rn}\left((1-\eta_R)v_m\right)(x-ty)H(y)dy.
		\edis
		Notice that $v_m\in\mathcal{L}_s(\Rn)$ (see \eqref{eq:l-very-weak}) and $(1-\eta_R)v_m$ is continuous on $\Rn$. The function $V^2_m\in H^1(t^{1-2s},\B_{2R}^+)$ satisfies 
		\bdis
		\left\lbrace
		\begin{array}{ll}
			M_sV^2_m=0 & \textrm{ in } \B_{2R}^+,\\
			V^2_m=(1-\eta_R)v_m=0 & \textrm{ on } B_{2R}.
		\end{array}
		\right.
		\edis		
		Let $\widetilde{V}^2_m(t,x):=V^2_m(-t,x)$ be the even reflection of $V^2_m$, then we have that
		$$
			\div\left(\vert t\vert^{1-2s}\nabla \widetilde{V}^2_m\right)=0 \textrm{ in } \B_{2R},
		$$
		in the sense of distribution. Applying the result in \cite[Proposition 2.1]{caff-salsa-silv}, we find that there exists a positive constant $C=C(N,s)$ and $\alpha\in(0,\beta)$ such that
		\be\label{eq:v2mhold}
			\parallel \widetilde{V}^2_m\parallel_{C^{0,\alpha}(\overline{\B}_R)}\leq \frac{C}{R^\alpha}\parallel \widetilde{V}_m^2\parallel_{L^{\infty}(\B_{2R})} \leq \frac{2C}{R^\alpha}\parallel V_m^2\parallel_{L^{\infty}(\B_{2R}^+)}.
		\ee
		Let us now estimate $\parallel V_m^2\parallel_{L^\infty(\B_{2R}^+)}$. We put $v_m^R:=(1-\eta_R)v_m$. Then, for $(t,x)\in \B_{2R}^+$, we have
		\ben\label{eq:sum2i}
			\vert V_m^2(t,x)\vert &=& C\left\lvert \int_{\Rn}\dfrac{v_m^R(x-ty)}{(1+\vert y\vert^2)^{\frac{N+2s}{2}}}dy\right\rvert,\nonumber\\
			& = & C\int_{\vert y\vert\leq 1}\dfrac{\vert v_m^R(x-ty)\vert}{(1+\vert y\vert^2)^{\frac{N+2s}{2}}}dy+C\sum_{i=0}^{+\infty}\int_{2^i\leq\vert y\vert\leq 2^{i+1}}\dfrac{\vert v_m^R(x-ty)\vert}{(1+\vert y\vert^2)^{\frac{N+2s}{2}}}dy,\nonumber\\
			& \leq & C \parallel v_m^R\parallel_{L^\infty(B_{4R})}+C\sum_{i=0}^{+\infty}\parallel v_m^R\parallel_{L^\infty(\B_{2^{i+4}R})}\int_{2^i\leq\vert y\vert\leq 2^{i+1}}\dfrac{dy}{(1+\vert y\vert^2)^{\frac{N+2s}{2}}},\nonumber\\
			& \leq & C \parallel v_m^R\parallel_{L^\infty(B_{4R})}+C\sum_{i=0}^{+\infty}(2^{i+4}R)^{\beta}\int_{2^i\leq\vert y\vert\leq 2^{i+1}}\dfrac{dy}{\vert y\vert^{N+2s}},\nonumber\\
			& \leq & CR^\beta + CR^{\beta}\sum_{i=0}^{+\infty}2^{(i+4)\beta}2^{-i2s}\int_{1\leq\vert z\vert\leq 2}\dfrac{dz}{\vert z\vert^{N+2s}},\nonumber\\
			& \leq & CR^\beta + CR^{\beta}\sum_{i=0}^{+\infty}2^{i(\beta-2s)}, \\
			&\leq & CR^\beta,\nonumber
		\een
		where we have used \eqref{eq:vmbound}, the change of variable $y=2^iz$ and the fact that $\beta<2s$ in \eqref{eq:sum2i} so that the summation is finite. It follows that
			\be\label{eq:v2mbound}
				\parallel V^2_m\parallel_{L^\infty(\overline{\B_{R}^+})}\leq C(N,s)R^\beta.
			\ee
		Using \eqref{eq:v2mbound} in \eqref{eq:v2mhold}, we get
		\be\label{eq:v2mholder}
			\parallel V^2_m\parallel_{C^{0,\alpha}(\overline{\B_R^+})}\leq \frac{1}{2}\parallel\widetilde{V}^2_m\parallel_{C^{0,\alpha}(\overline{\B}_R)}\leq C(R).
		\ee
		Since $V_m=V_m^1+V_m^2$, we obtain
		\bdis
			\parallel V_m\parallel_{L^{\infty}(\overline{\B_R^+})}\leq CR^\beta \quad\textrm{ for every } R>1, \textrm{ with } r_mR\leq 1,
		\edis
		by \eqref{eq:vm1bound} and \eqref{eq:v2mbound}. Using \eqref{eq:vm1holder} and \eqref{eq:v2mholder}, we also have
		\bdis
			\parallel V_m\parallel_{C^{0,\alpha}(\overline{\B_R^+})}\leq C(R).
		\edis
		We then conclude that, up to a subsequence, the sequence $\left(V_m\right)$ converges uniformly to some function $V$ on compact subsets of $\overline{\RN}$ by Arzel\~a-Ascoli theorem. Recall that
		\bene
			V_m(t,x) &=& C\int_{\Rn}\dfrac{v_m(x-ty)}{\left(1+\vert y\vert^2\right)^{\frac{N+2s}{2}}}dy,\\
			&=& C\int_{\vert y\vert\leq 1}\dfrac{v_m(x-ty)}{(1+\vert y\vert^2)^{\frac{N+2s}{2}}}dy+C\sum_{i=0}^{+\infty}\int_{2^i\leq\vert y\vert\leq 2^{i+1}}\dfrac{v_m(x-ty)}{(1+\vert y\vert^2)^{\frac{N+2s}{2}}}dy.
		\eene
		Since $\parallel V_m\parallel_{L^\infty(\B_1^+)}$ is bounded then, by the dominated convergence theorem,
		\bdis
			\int_{\vert y\vert \leq 1} \dfrac{v_m(x-ty)}{(1+\vert y\vert^2)^{\frac{N+2s}{2}}}dy~~~\to~~~	K\int_{\vert y\vert \leq 1} \dfrac{\bar{\delta}_{x_0,\nu}^s(x-ty)}{(1+\vert y\vert^2)^{\frac{N+2s}{2}}}dy,
		\edis
		as $m\to +\infty$, recall that $v_m\to K\bar{\d}_{x_0,\nu}^s$ uniformly on compact subsets of $\Rn$. Now put
		$$
		A_m^i(t,x):=\int_{2^i\leq\vert y\vert\leq 2^{i+1}}\dfrac{v_m(x-ty)}{(1+\vert y\vert^2)^{\frac{N+2s}{2}}}dy.
		$$
		We now prove that $V_m\to \cP(t,\cdot)\star\bar{\d}_{x_0,\nu}^s$ uniformly on compact subsets of $\ov\RN$.
		Since $v_m~\to~K\bar{\delta}_{x_0,\nu}$ uniformly on compact subsets of $\Rn$ then, by the dominated convergence theorem, we have that
		\bdis
			\lim_{m\to+\infty}A_m^i(t,x)=\int_{2^i\leq\vert y\vert\leq 2^{i+1}}\lim_{m\to+\infty}\left(\dfrac{v_m(x-ty)}{(1+\vert y\vert^2)^{\frac{N+2s}{2}}}\right)dy=\int_{2^i\leq\vert y\vert\leq 2^{i+1}}\dfrac{\bar{\delta}_{x_0,\nu}^s(x-ty)}{(1+\vert y\vert^2)^{\frac{N+2s}{2}}}dy.
		\edis
		Let $r>0$ and $z=(t,x)\in \B_r^+$ fixed. With similar arguments as in \eqref{eq:sum2i}, we have that
		\bdis
			\vert A_m^i(z)\vert\leq C(r)\sum_{i=0}^{+\infty}2^{i(\beta-2s)}\leq C(r),
		\edis
		since $\beta-2s<0$. Consequently, by the dominated convergence theorem,
		\bene
			\lim_{m\to +\infty}\sum_{i=0}^{+\infty} A^i_m(z) &=& \sum_{i=0}^{+\infty}\lim_{m\to +\infty} A^i_m(z),\\
			&=& K\sum_{i=0}^{+\infty} \int_{2^i\leq\vert y\vert\leq 2^{i+1}}\dfrac{\bar{\delta}_{x_0,\nu}^s(x-ty)}{(1+\vert y\vert^2)^{\frac{N+2s}{2}}}dy,\\
			&=& K\int_{\vert y\vert\geq 1}\dfrac{\bar{\delta}_{x_0,\nu}^s(x-ty)}{(1+\vert y\vert^2)^{\frac{N+2s}{2}}}dy.
		\eene
		Finally, for every $z=(t,x)\in \B_r^+$, we conclude that
		\bene
			V(t,x) &=& K\int_{\vert y\vert\leq 1}\dfrac{\bar{\delta}_{x_0,\nu}^s(x-ty)}{(1+\vert y\vert^2)^{\frac{N+2s}{2}}}dy + K\int_{\vert y\vert\geq 1}\dfrac{\bar{\delta}_{x_0,\nu}^s(x-ty)}{(1+\vert y\vert^2)^{\frac{N+2s}{2}}}dy,\\
			&=& K\int_{\Rn}\dfrac{\bar{\delta}_{x_0,\nu}^s(x-ty)}{(1+\vert y\vert^2)^{\frac{N+2s}{2}}}dy,\\
			&=& K\cH_+^{x_0,\nu}(t,x).
		\eene
		Since $r$ is arbitrary, we get the desired result.
	\end{proof}
	
\section{Regularity up to the Dirichlet-Neumann interface}\label{Sect6}
		Note that the estimates in the following lemmas hold in a tubular neighbourhood of the interface set $\de\O$, the boundary of $\O$. We define
	\be\label{eq:Hplus}
		\Hplus(t,x):={h^+}(t,\d(x)),~~~\forall (t,x)\in(0,+\infty)\times\O.
	\ee
	In this section, we assume that for any $\bar{x}_0\in\O\cap B_{1/2}$, there exists a unique $x_0\in\de\O$ such that $\vert \bar{x}_0-x_0\vert=\d_\O(\bar{x}_0).$
	We have the following  result comparing $ \Hplus $ and $ \Hom$ in $\B_r^+(\bar{x}_0)$, where $r=\frac{1}{2}\vert \bar{x}_0-x_0\vert$, see \eqref{eq:ball}.
\begin{lemma}\label{l:estimates}

Let ${\bar x}_0\in\O\cap B_{1/2}$ and $x_0\in\partial\O$ be such that $\vert \bar{x}_0-x_0\vert=\delta_\O({\bar x}_0)$. Let $r:=\frac{\d_\O(\bar{x}_0)}{2}$. Then 
\be\label{eq:est1}
	\| \Hom-\H_{\O}^+ \|_{L^\infty(\B_r^+({\bar x}_0))} \leq C r^{2s},
\ee
\be\label{eq:est2}
	\left[\Hom-\H_{\O}^+\right]_{  C^{s-\e}( \B_r^+({\bar x}_0) )}\leq Cr^s
\ee
 and 
\be\label{eq:est3}
 	\left[\left(\Hom\right)^{-1}\right]_{  C^{s-\epsilon}( \B_r^+({\bar x}_0))}\leq C r^{-2s+\epsilon},
\ee
where the positive relabelled constant $C$ depends only on $N,~s,~\e$ and $\O$.
\end{lemma}
\begin{proof}
	To simplify the notations, we define
	$$
		\bar{\d}(x):=\bar{\d}_{x_0,\nu}(x)=\left[(x-x_0)\cdot\nu(x_0)\right]_+, \qquad \d(x):=\d_\O(x)
	$$
	and we recall
	\bdis
		h^+(t,\d)= C (t^2+\d^2)^{s/2}\cos^{2s}(\th/2)\,  {}_2F_1\left(0, 1;1-s; \frac{1-\cos(\th)}{2}\right),
	\edis
	where $\th=\textrm{arctan}(\frac{t}{\d(x)})\in(0,\pi/2)$. Since $(t,x)\in\B_r^+({\bar x}_0)\subset[0,r)\times\O$, then $\th/2\in(0,\pi/4)$ and thus $\frac{\sqrt{2}}{2}<\cos(\th/2)<1$.\\
	On the other hand, we have that $1>\frac{1-\cos(\th)}{2}\to 0$ as $\th\to 0$. Hence by \eqref{eq:power-series},
	$$
	{}_2F_1\left(0, 1;1-s; \frac{1-\cos(\th)}{2}\right)=a_0+O(1-\cos\th),\qquad 1-\cos(\th)\to 0 \textrm{ as } \th\to 0.
	$$
	Therefore for $(t,x)\in\B_r^+({\bar x}_0)$, there exist two positive constants $C_1\leq C_2$ such that
	\bdis
		C_1 \left(t^2+\d^2(x)\right)^{s/2} \leq h^+(t,\d(x)) \leq C_2 \left(t^2+\d^2(x)\right)^{s/2}
	\edis
	and consequently
	\be\label{eq:gradhplus}
		\sup_{(t,x)\in\B_r^+({\bar x}_0)}\vert\nabla h^+(t,\d(x))\vert\leq Cr^{s-1}.
	\ee
	Note that for $x\in B_r({\bar x}_0)$, we have
	\bdis
		\vert \bar{\d}(x)-\d(x)\vert\leq Cr^2,
	\edis
	since $\de\O$ is $C^{1,1}$.\\
	Now, to prove \eqref{eq:est1}, we write
	\bene
		\left\lvert \Hom(t,x)-\H_{\O}^+(t,x)\right\rvert &=& \left\lvert h^+(t,\bar{\d}(x))-h^+(t,\d(x))\right\rvert,\\
		&=& \left\lvert\int_0^1\pd{h^+}{\tau}(t,\tau\bar{\d}(x)+(1-\tau)\d(x))d\tau (\bar{\d}(x)-\d(x))\right\rvert,\\
		& \leq & \sup_{(t,x)\in\B_r^+({\bar x}_0)}\vert\nabla h^+(t,\d(x))\vert \vert \bar{\d}(x)-\d(x)\vert,\\
		& \leq & Cr^{s-1}r^2=Cr^{1+s}\leq Cr^{2s}.
	\eene
	To see \eqref{eq:est2}, we define
	\bdis
		G_+(z):=\Hom(z)-\H_{\O}^+(z), \quad\textrm{ for $z=(t,x)\in\B_r^+({\bar x}_0)$}
	\edis
	and for $\tau\in(0,1)$,
	\bdis
		g^+_{\tau}(t,x)=\pd{h^+}{\tau}(t,\tau\bar{\d}(x)+(1-\tau)\d(x)).
	\edis
	Let $z_1,z_2\in\B_r^+({\bar x}_0)$, we have
	\bene
		\dfrac{\vert G_+(z_1)-G_+(z_2)\vert}{\vert z_1-z_2\vert^{s-\epsilon}} &=& \frac{1}{\vert z_1-z_2\vert^{s-\epsilon}}\left\lvert \int_0^1g^+_{\tau}(t_1,x_1)d\tau\left(\bar{\d}(x_1)-\d(x_1)\right)\right.\\
		& - & \left.\int_0^1g^+_{\tau}(t_2,x_2)d\tau\left(\bar{\d}(x_2)-\d(x_2)\right)\right\rvert,\\
		& \leq & \frac{1}{\vert z_1-z_2\vert^{s-\epsilon}}\vert\bar{\d}(x_1)-\d(x_1)\vert \left\lvert\int_0^1 g^+_{\tau}(z_1)-g^+_{\tau}(z_2)d\tau \right\rvert\\
		& + & \left(\vert \bar{\d}(x_1)-\d(x_1)\vert+\vert \bar{\d}(x_2)-\d(x_2)\vert\right)\frac{1}{\vert z_1-z_2\vert^{s-\epsilon}}\left\lvert\int_0^1g^+_{\tau}(z_2)d\tau \right\rvert,\\
		& \leq & Cr^{-s+\epsilon}r^2r\left\lvert\int_0^1\int_0^1 \pd{g^+_{\tau}}{\rho}(\rho z_1-(1-\rho)z_2)d\rho d\tau\right\rvert\\
		& + & \left\lvert C(r^2+r^2)r^{-s+\epsilon}\int_0^1g^+_{\tau}(z_2)d\tau \right\rvert,\\
		& \leq & Cr^{3-s+\epsilon}r^{s-2}+Cr^{2-s+\epsilon}r^{s-1},\\
		& \leq & Cr^{1+\epsilon}\leq Cr^s.
	\eene
	Finally, we prove \eqref{eq:est3}. We have, for $z_1,z_2\in\B_r^+({\bar x}_0)$,
	\bene
		\dfrac{\left\lvert\dfrac{1}{\Hom(z_1)}-\dfrac{1}{\Hom(z_2)}\right\rvert}{\vert z_1-z_2\vert^{s-\epsilon}}&=& \dfrac{\left\lvert\Hom(z_1)-\Hom(z_2)\right\rvert}{\vert z_1-z_2\vert^{s-\epsilon}}\dfrac{1}{\Hom(z_1)\Hom(z_2)},\\
		& \leq & Cr^sr^{-s+\e}r^{-2s} \leq Cr^{\epsilon-2s},
	\eene
	where we used the fact that $\frac{\sqrt{2}}{2}<\cos(\th/2)<1$ and ${}_2F_1\left(0, 1;1-s; \frac{1-\cos(\th)}{2}\right)\simeq a_0>0$ in $\B^+_r({\bar x}_0)$.
\end{proof}

\begin{lemma}\label{l:estimates2}

Let ${\bar x}_0\in\O\cap B_{1/2}$ and $x_0\in\partial\O$ as in Lemma \ref{l:estimates}. Then there exists a positive universal constant $C$ such that
\be\label{eq:est21}
	\| d^s-\H_{\O}^+ \|_{L^\infty(\B_r^+({\bar x}_0))} \leq C r^{2s},
\ee

\be\label{eq:est22}
	\left[ d^s-\H_{\O}^+\right]_{  C^{s-\e}( \B_r^+({\bar x}_0) )}\leq C r^s
\ee
and

\be\label{eq:est23}
 	\left[ d^{-s} \right]_{C^{s-\e}( \B_r^+(\bar{x}_0))}\leq C r^{-2s+\e},
\ee
where $d(t,x)=d_\O(t,x)=\left(t^2+\d_\O^2(x)\right)^{1/2}$ and the positive constant $C$ depends only on $N,~s,~\e$ and $\O$.
\end{lemma}
\begin{proof}
	Define $d_+(t,\d(x)):=d(t,x)=\left(t^2+\d^2(x)\right)^{1/2}.$ Hence, we have that
	\bene
		\left\lvert d^s(t,x)-\H_{\O}^+(t,x)\right\rvert &=& \left\lvert d_+^s(t,\d(x))-h^+(t,\d(x))\right\rvert,\\
		& \leq & \left\lvert d_+^s(t,\d(x))-d_+^s(t,\bar{\d}(x))\right\rvert + \left\lvert d_+^s(t,\bar{\d}(x)) - h^+(t,\d(x))\right\rvert,\\
		& = & I~~~ + ~~~II.
	\eene
	To prove estimates \eqref{eq:est21} and \eqref{eq:est22} for the quantities $I$ and $II$, we use the same argument as in Lemma \ref{l:estimates} by remarking that for $I$
	\bdis
		\left\lvert \nabla d_+^s(t,\d)\right\rvert\leq r^{s-1}.
	\edis
	For the quantity $II$, we note that there are two positive constants $C_1$ and $C_2$ such that
	\bdis
		C_1 d_+^s(t,\bar{\d}) \leq h^+(t,\bar{\d}) \leq C_2 d_+^s(t,\bar{\d}),
	\edis
	similarly as in the proof of Lemma \ref{l:estimates}, see also the definition of $h^+$ in Section \ref{Sect3}.\\	
	For \eqref{eq:est23}, recall first that for $z_1,z_2\in\B_r^+(\bar{x_0})$,
	$$
		\vert z_1-z_2\vert\leq Cr,\quad d^{-s}(z_1)\leq Cr^{-s} \textrm{ and } \sup_{z\in\B^+_r(\bar{x_0})}\vert\nabla d^s(z)\vert\leq Cr^{s-1}.
	$$
	We have
	\bdis
		\frac{1}{d^s}(z_1)-\frac{1}{d^s}(z_2)=\left(d^s(z_2)-d^s(z_1)\right)d^{-s}(z_1)d^{-s}(z_2).
	\edis
	Then, we obtain
	\bene
		\dfrac{\vert d^{-s}(z_1)-d^{-s}(z_2)\vert}{\vert z_1-z_2\vert^{s-\epsilon}} &=& \left\lvert\int_0^1\nabla_\tau d^s(\tau z_1+(1-\tau)z_2)d\tau\right\rvert \frac{\vert z_1-z_2\vert}{\vert z_1-z_2\vert^{s-\epsilon}}d^{-s}(z_1)d^{-s}(z_2),\\
		& \leq & C\sup_{z\in\B^+_r(\bar{x_0})}\vert\nabla d^s(z)\vert r^{1-s+\epsilon}r^{-2s},\\
		& \leq & Cr^{s-1}r^{1-3s+\e}=Cr^{-2s+\e},
	\eene
	up to relabeling the positive constant $C$ that depends only on $N$, $s$, $\e$ and $\O$.
\end{proof}

\begin{lemma}\label{l:sharpregul}
Let ${\bar x}_0\in\O\cap B_{1/2}$ and $x_0\in\partial\O$ be the unique point such that $2r:=\vert x_0-{\bar x}_0\vert=\d({\bar x}_0)$. Assume that
	\be\label{eq:linfbound}
		\Vert W-Q(x_0)\Hom\Vert_{L^\infty(\B_{2r}^+(x_0))}\leq Cr^{2s-\epsilon},\qquad \textrm{ with } \vert Q(x_0)\vert\leq C.
	\ee
Then
	$$
		\left[ W-Q(x_0)\Hom\right]_{C^{s-\epsilon}(\B_r^+({\bar x}_0))}\leq CC_0r^s,
	$$
	for some constant $C>0$ depending only on $N$, $s$ and $\e$.
\end{lemma}
\begin{proof}
		Set 
		\bdis
			V_r(z):=\dfrac{W(x_0+rz)-Q\Hom(x_0+rz) }{r^s}.
		\edis
		Since $[0,r]\times (x_0+rB_2)\subset \B_{4}^+$, by \eqref{eq:linfbound}, we have that
		\bdis
			\Vert V_r\Vert_{L^\infty(\B_{2}^+)({\bar x}_0)} \leq Cr^{s-\epsilon}.
		\edis
		Furthermore, we have
		\bdis
			M_sV_r=0 \qquad \textrm{ in  } \B_2^+({\bar x}_0)
		\edis
		and
		\bdis
			\lim_{t\rightarrow0}N_sV_r(t,x)=r^sf(x_0+rx),\quad x\in B_2({\bar x}_0).
		\edis
		Then, by Theorem \ref{t:interegul},
		\bdis
			\left[ V_r\right]_{C^{s-\epsilon}(\B_1^+({\bar x}_0))}\leq Cr^{s-\epsilon}.
		\edis
		Therefore, we infer that
		$$
			\left[ W-Q\Hom\right]_{C^{s-\epsilon}(\B_r^+({\bar x}_0))}=r^sr^{-s+\epsilon}\left[V_r\right]_{C^{s-\epsilon}(\B_1^+({\bar x}_0))}\leq Cr^sr^{-s+\epsilon}r^{s-\epsilon}=Cr^s,
		$$
		as desired.\\
	\end{proof}

We now prove the following result.
	\begin{proposition}\label{p:II}
		Let $\O\subset\Rn$ be a bounded domain of class $C^{1,1}$ and $W$ satisfy equation \eqref{eq:pbm1}. Then $\frac{W}{{\cH_{\O}^+}}\in C^{s-\epsilon}(\overline{\B_r^+}({\bar x}_0))$ for ${\bar x}_0$ and $r$ as in Lemma \ref{l:sharpregul}. Moreover, we have the following estimate
		\bdis
			\left[\dfrac{W}{\cH_{\O}^+}\right]_{C^{s-\epsilon}(\overline{\B_r^+}({\bar x}_0))}\leq C\left(\parallel W\parallel_{L^\infty(\RN)} + \parallel f\parallel_{L^\infty(\Rn)}\right),
		\edis
		where the positive constant $C$ depends only on $N$, $s,~\e$ and $\O$.
	\end{proposition}
	\begin{proof}
		By \eqref{eq:est1} and Proposition \ref{p:pbm}, we have that
		\be\label{eq:sharpuniform}
			\parallel W-Q(x_0)\cH_{\O}^+\parallel_{L^\infty(\B_{2r}^+({\bar x}_0))}\leq Cr^{2s-\epsilon}.
		\ee
		Lemma \ref{l:sharpregul} and \eqref{eq:est2} yield
		\be\label{eq:sharpreg}
			\left[ W-Q(x_0)\cH_{\O}^+\right]_{C^{s-\e}(\B_{2r}^+({\bar x}_0))}\leq Cr^s,
		\ee
		for $r$ as in Lemma \ref{l:sharpregul}. Now, for $z_1,z_2\in \B_{2r}^+({\bar x}_0)\supset\ov{\B_r^+}({\bar x}_0)$, we have
		\bene
			\dfrac{W}{\cH_{\O}^+}(z_1)-\dfrac{W}{\cH_{\O}^+}(z_2) &=& \dfrac{\left(W-Q(x_0)\cH_{\O}^+\right)(z_1)-\left(W-Q(x_0)\cH_{\O}^+\right)(z_2)}{\cH_{\O}^+(z_1)}\\
			&+& \left(W-Q(x_0)\cH_{\O}^+\right)(z_2)\left[\left(\cH_{\O}^+\right)^{-1}(z_1)-\left(\cH_{\O}^+\right)^{-1}(z_2)\right].
		\eene		
		On one hand, using \eqref{eq:sharpreg}, we obtain
		\ben\label{eq:WH1}
			\left\vert \dfrac{\left(W-Q(x_0)\cH_{\O}^+\right)(z_1)-\left(W-Q(x_0)\cH_{\O}^+\right)(z_2)}{\cH_{\O}^+(z_1)}\right\vert & \leq & Cr^s\left(\cH_{\O}^+\right)^{-1}(z_1)\vert z_1-z_2\vert^{s-\e}\nonumber,\\
			& \leq & C\vert z_1-z_2\vert^{s-\e},
		\een
		by noting that $\cH_{\O}^+ \sim r^s$ in $\B_{2r}^+({\bar x}_0)$, up to relabeling the positive constant $C$.\\
		On the other hand, by \eqref{eq:sharpuniform} and \eqref{eq:est3}, we infer that
		\ben\label{eq:WH2}
			\left\lvert\left(W-Q(x_0)\cH_{\O}^+\right)(z_2)\right\rvert\left\lvert\left(\cH_{\O}^+\right)^{-1}(z_1)-\left(\cH_{\O}^+\right)^{-1}(z_2)\right\rvert & \leq & Cr^{2s-\e}\left\lvert\left(\cH_{\O}^+\right)^{-1}(z_1)-\left(\cH_{\O}^+\right)^{-1}(z_2)\right\rvert\nonumber,\\
			& \leq & Cr^{2s-\e}r^{-2s+\e}\vert z_1-z_2\vert^{s-\e}\nonumber,\\
			&=& C\vert z_1-z_2\vert^{s-\e}.
		\een
		Therefore, by \eqref{eq:WH1} and \eqref{eq:WH2},
		\bdis
			\left[\dfrac{W}{\cH_{\O}^+}\right]_{C^{s-\e}(\ov{\B_r^+}({\bar x}_0))}\leq C.
		\edis
	\end{proof}
	
	\section{Proof of Theorem \ref{t:I}}\label{Section7}
	
		\textbf{Regularity of Set :}
				Let $k\in\N$ and $\alpha\in(0,1]$. A set $\Omega\subset\Rn$ is of class $C^{k,\alpha}$ if there exists $M>0$ such that for any $x_0\in\partial\Omega$, there exist a ball $B=B_r(x_0),r>0$ and an isomorphism $\phi:Q\longrightarrow B$ such that :
				\medskip
				
				\begin{displaymath}
					\phi\in C^{k,\alpha}(\ov{Q}),\quad \phi^{-1}\in C^{k,\alpha}(\ov{B}),\quad
					\phi(Q_+)=B\cap\Omega, \quad \phi(Q_0)=B\cap\partial\Omega
				\end{displaymath}
				
				and
				\begin{displaymath}
					\parallel \phi\parallel_{C^{k,\alpha}(\ov{Q})}+\parallel \phi^{-1}\parallel_{C^{k,\alpha}(\ov{B})}\leq M.
				\end{displaymath}
				where $Q$ is a cylinder
				\bdis
					Q :=\lbrace x=(x',x_N)\in\R^{N-1}\times\R:\vert x'\vert<1  \text{ and } \vert x_N\vert<1\rbrace,
				\edis
				\bdis
					Q_+ :=\lbrace x=(x',x_N)\in\R^{N-1}\times\R:\vert x'\vert<1 \text{ and } 0<x_N<1\rbrace
				\edis
				and
				\bdis
					Q_0 :=\lbrace x\in Q: x_N=0\rbrace,
				\edis
			see \cite[Section $1$]{DPE-hitch}.
	In order to complete the proof of Theorem \ref{t:I}, we will need the following result.
	\begin{proposition}\label{p:pbmII}
		Let $f\in L^\infty(\Rn)$, $\O$ be a bounded domain of class $C^{1,1}$ in $\Rn$ and $W\in\dot{H}^1(t^{1-2s},\RN)$ satisfy
		\bdis
			\left\lbrace
			\begin{array}{ll}
				M_sW=0 & \textrm{ in }\R^{N+1}_+,\\
				\lim_{t\rightarrow0}N_sW(t,\cdot)=f & \textrm{ on } \O,\\
				W=0 &  \textrm{ on } \Rn\setminus\O.
			\end{array}
			\right.
		\edis
		Then $\frac{W}{{\Hplus}}\in C^{s-\epsilon}([0,1]\times\overline{\O})$ and we have
		\bdis
			\left\lVert\dfrac{W}{\Hplus}\right\rVert_{C^{s-\e}([0,1]\times\overline{\O})}\leq C\left(\parallel W\parallel_{L^\infty(\RN)} + \parallel f\parallel_{L^\infty(\Rn)}\right),
		\edis
		where $C$ is a positive constant depending only on $N$, $s$, $\e$ and $\O$.
	\end{proposition}
 	\begin{proof}
		We use similar argument as in \cite[Proposition 1.1]{RS-bound}. We assume that $\Vert W\Vert_{L^\infty(\RN)}+\Vert f\Vert_{L^\infty(\Rn)}\leq 1$ and put $U=\dfrac{W}{\cH_{\O}^+}.$ Then by Proposition \ref{p:II}, we have
\bdis
	\dfrac{\vert U(\dot{z})-U(\dot w)\vert}{\vert\dot z-\dot w\vert^{s-\e}}\leq C,
\edis
for $\dot z=(\dot q,\dot x)$ and $\dot w=(\dot l,\dot y)$ such that  $\dot y\in B_{R/2}(\dot x)$ and $0\leq \dot q,\dot l<R/2$, where we define $R:=\delta(\dot x)$. Here $C$ is a positive constant depending only on $s$ and $\O$.\\
Our aim is to show that
\bdis
	\left[ U\right]_{C^{s-\e}([0,1]\times\overline{\O})}\leq C.
\edis
Indeed, since $\O$ has $C^{1,1}$ boundary by assumption, we can flatten the boundary of $\O$ in a neighbourhood of a point $x_0\in\de\O$. Thus, there exist a constant $\rho_0>0$ small enough and a $C^{1,1}$-diffeomorphism $\psi$ from $B_{\rho_0}(x_0)$ to $Q$ such that
$$
	\psi(\O\cap B_{\rho_0}(x_0))=\left\lbrace (x',x_N)\in B_1:x_N>0\right\rbrace, \quad \psi(\de\O\cap B_{\rho_0}(x_0))=\left\lbrace (x',0):~\vert x'\vert<1\right\rbrace
$$
and $\quad \psi(\d(x))=x_N,$ where $\d(x)=$dist$(x,\de\O)$ and $x\in\O$. We now let 
$$
	\phi(t,x):=(t,\psi(x)),\quad \textrm{for } (t,x)\in[0,\rho_0)\times(B_{\rho_0}(x_0)\cap\O),
$$
and denote $\overline{U}:=U\circ\phi^{-1}$.\\
Let $z=(q,x)$ and $w=(l,y)$ such that  $y\in B_{x_N/2}(x)$ and $0\leq q,l<x_N/2$, then $\phi^{-1}(\cdot,y)\in B_{R/2}(\phi^{-1}(\cdot,x))$. We thus have that
\bene
	\dfrac{\vert \overline{U}(z)-\overline{U}(w)\vert}{\vert z-w\vert^{s-\e}} &=& \dfrac{\vert U\circ\phi^{-1}(z)-U\circ\phi^{-1}(w)\vert}{\vert z-w\vert^{s-\e}},\\
	&=& \dfrac{\vert U\circ\phi^{-1}(z)-U\circ\phi^{-1}(w)\vert}{\vert \phi^{-1}(z)-\phi^{-1}(w)\vert^{s-\e}}\times\left(\dfrac{\vert \phi^{-1}(z)-\phi^{-1}(w)\vert}{\vert z-w\vert}\right)^{s-\e}.
\eene
Since $\phi^{-1}(\cdot,y)\in B_{R/2}(\phi^{-1}(\cdot,x))$, it is plain that
\bdis
	\dfrac{\vert U\circ\phi^{-1}(z)-U\circ\phi^{-1}(w)\vert}{\vert \phi^{-1}(z)-\phi^{-1}(w)\vert^{s-\e}}\leq C.
\edis
It is clear that
\bdis
	\vert \phi^{-1}(z)-\phi^{-1}(w)\vert \leq L \vert z-w\vert,
\edis
with $L>0$ depends only on $\O$. For any $z=(q,x)$ and $w=(l,y)$ such that  $y\in B_{x_N/2}(x)$ and $0\leq q,l<x_N/2$, we finally get that 
\be\label{eq:U1}
	\dfrac{\vert \overline{U}(z)-\overline{U}(w)\vert}{\vert z-w\vert^{s-\e}} \leq L^{s-\e} C.
\ee
	We note that \eqref{eq:U1} holds for any $z,w$ such that $\vert z-w\vert\leq \zeta x_N$, where $\zeta\in(0,\frac{\sqrt{2}}{2})$ depends on $\O$.\\
Now let $z=(q,z',z_N)$ and $w=(l,w',w_N)$ be two points in $\B_{1/8}^{++}:=[0,1/8)\times\left(\{x_N>0\}\cap B_{1/8}\right)$. We put $r=\vert z-w\vert$, $\bar{z}=(q,z',z_N+r)$ and $\bar{w}=(l,w',w_N+r)$. We also set $w_k=(1-\zeta^k)\bar{w}+\zeta^kw$ and $z_k=(1-\zeta^k)z+\zeta^k\bar{z}$, for $k\geq  1$. Thus, for $\zeta\in(0,\frac{\sqrt{2}}{2})$,
\bene
	\vert z_{k+1}-z_k\vert &=& \left\lvert (1-\zeta^{k+1})z+\zeta^{k+1}\bar{z}-\left[(1-\zeta^k)z+\zeta^k\bar{z}\right]\right\rvert,\\
	&=& \vert \left(1-\zeta^{k+1}-1-\zeta^k\right)z-\left(\zeta^{k+1}-\zeta^k\right)\bar{z}\vert,\\
	&=& \left(1-\zeta\right)\zeta^k\vert(q,z',z_N)-(q,z',z_N+r)\vert,\\
	& \leq & \zeta r <  \zeta x_N
\eene
and similarly
\bdis
	\vert w_{k+1}-w_k\vert\leq \zeta r\leq \zeta x_N.
\edis
Using \eqref{eq:U1}, we have that
\be\label{eq:U2}
	\vert \overline{U}(z_{k+1})-\overline{U}(z_k)\vert\leq C\vert z_{k+1}-z_k\vert^{s-\e}\leq Cr^{s-\e}
\ee
and
\be\label{eq:U3}
	\vert \overline{U}(w_{k+1})-\overline{U}(w_k)\vert \leq Cr^{s-\e}.
\ee
Recall that $r=\vert z-w\vert=\vert \bar{z}-\bar{w}\vert$. We put
$$
\bar{h}_l:=(1-\mu^l)\bar{z}+\mu^l\bar{w},
$$
 where $\mu\in(0,\zeta),~l=1,2,...,M$, with $\bar{h}_0=\bar{w}$ and $\bar{h}_{M+1}=\bar{z}$. We then have 
$$
\vert \bar{h}_{l+1}-\bar{h}_l\vert\leq \zeta r<\zeta x_N
$$
and thus
\be\label{eq:U4}
	\vert \overline{U}(\bar{z})-\overline{U}(\bar{w})\vert \leq \sum_{l=1,...,M}\vert \overline{U}(\bar{h}_{l+1})-\overline{U}(\bar{h}_l)\vert \leq C\vert \bar{h}_{l+1}-\bar{h}_l\vert^{s-\e} \leq C\vert \bar{z}-\bar{w}\vert^{s-\e}\leq Cr^{s-\e}.
\ee
By \eqref{eq:U2}, \eqref{eq:U3} and \eqref{eq:U4}, we finally get
\bene
	\vert \overline{U}(z)-\overline{U}(w)\vert & \leq & \sum_{k\geq 1}\vert \overline{U}(z_{k+1})-\overline{U}(z_k)\vert + \vert \overline{U}(\bar{z})-\overline{U}(\bar{w})\vert + \sum_{k\geq 1}\vert \overline{U}(w_{k+1})-\overline{U}(w_k)\vert,\\
	& \leq & C\sum_{k\geq 1}\zeta^{k(s-\e)}\vert z-\bar{z}\vert^{s-\e} + C\vert \bar{z}-\bar{w}\vert^{s-\e} + C\sum_{k\geq 1}\zeta^{k(s-\e)}\vert w-\bar{w}\vert^{s-\e},\\
	& \leq & Cr^{s-\e}\sum_{k\geq 1}\zeta^{k(s-\e)} + Cr^{s-\e}\leq C\vert z-w\vert^{s-\e},
\eene
up to relabeling the positive constant $C$. Therefore,
\bdis
	\left[\overline{U}\right]_{C^{s-\e}(\B_{1/8}^{++})}\leq C.
\edis
	Thus by compactness of $\ov\O$, we then deduce that
	\bdis
	\left[U\right]_{C^{s-\e}([0,1]\times\ov{\O})}\leq C.
	\edis
	By adding the $L^\infty$ bound, we have that $U\in C^{s-\e}([0,1]\times\ov{\O})$ and
	\bdis
		\Vert U\Vert_{C^{s-\e}([0,1]\times\ov{\O})}\leq C,
	\edis
	where $C$ depends only on $N$, $s,~\e$ and $\O$.
 	\end{proof}
\begin{proof}{ of \bf Theorem \ref{t:I}}
	By Proposition \ref{p:pbmII}, that is $\frac{W}{{\Hplus}}\in C^{s-\e}([0,1]\times\overline{\O})$, it suffices to prove that $\frac{\cH_{\O}^+}{d^s}\in C^{s-\e}([0,1]\times\overline{\O}).$
	For $z_1,z_2\in\B_r^+({\bar x}_0)$, we decompose
	\bene
		\dfrac{\cH_{\O}^+}{d^s}(z_1)-\dfrac{\cH_{\O}^+}{d^s}(z_2) &=& \dfrac{\left(\cH_{\O}^+-d^s\right)(z_1)-\left(\cH_{\O}^+-d^s\right)(z_2)}{d^s(z_1)}\\
		&+& \left(\cH_{\O}^+-d^s\right)(z_2)\left[d^{-s}(z_1)-d^{-s}(z_2)\right].
	\eene
	By \eqref{eq:est22}, we have
	\bdis
		 \dfrac{\left\lvert\left(\cH_{\O}^+-d^s\right)(z_1)-\left(\cH_{\O}^+-d^s\right)(z_2)\right\rvert}{d^s(z_1)}\leq \frac{r^s}{d^s(z_1)}C\vert z_1-z_2\vert^{s-\e}\leq C\vert z_1-z_2\vert^{s-\e}.
	\edis
	Also by \eqref{eq:est21} and \eqref{eq:est23}, we obtain that 
	\bene
		\left\lvert (\cH_{\O}^+ -d^s)(z_2)\left[d^{-s}(z_1)-d^{-s}(z_2)\right]\right\rvert & \leq & Cr^{2s}r^{-2s+\e}\vert z_1-z_2\vert^{s-\e},\\
		& = & Cr^\e\vert z_1-z_2\vert^{s-\e},\\
		& \leq & C\vert z_1-z_2\vert^{s-\e}.
	\eene
	Going back to the decomposition, we deduce that
	\bdis
		\left[\dfrac{\cH_{\O}^+}{d^s}\right]_{C^{s-\e}(\B^+_r({\bar x}_0))}\leq C.
	\edis
	Using similar arguments as in the proof of Proposition \ref{p:pbmII}, we thus get
	\bdis
		\left[\dfrac{\cH_{\O}^+}{d^s}\right]_{C^{s-\e}([0,1]\times\overline{\O})}\leq C.
	\edis
	Noting that the $L^\infty$ bound follows from Lemma \ref{l:estimates2}, we deduce that
	\bdis
		\dfrac{\cH_{\O}^+}{d^s}\in C^{s-\e}([0,1]\times\overline{\O}).
	\edis
	Therefore
	\bdis
		\dfrac{W}{d^s}\in C^{s-\e}([0,1]\times\overline{\O})
	\edis
	and moreover
	\bdis
		\left\lVert\dfrac{W}{d^s}\right\rVert_{C^{s-\e}([0,1]\times\overline{\O})}\leq C\left(\Vert W\Vert_{L^\infty(\Rn)}+\Vert f\Vert_{L^\infty(\Rn)}\right),
	\edis
	as desired.
\end{proof}
 {99}
 
\end{document}